\documentclass[12pt]{article}

\usepackage[utf8]{inputenc}

\usepackage[english]{babel}

\usepackage[shortlabels]{enumitem}
\usepackage[T1]{fontenc}

\usepackage{tikz}
\usepackage{verbatim}

\usepackage{amsmath,amsfonts,amssymb,amsthm}
\usepackage{nccmath}
\usepackage{mathrsfs,indentfirst, latexsym}
\usepackage{amscd}
\usepackage{eucal}
\usepackage{fancyhdr}

\numberwithin{equation}{section}

\newtheorem{theorem}{Theorem}
\newtheorem{lemma}[theorem]{Lemma}
\newtheorem{proposition}[theorem]{Proposition}
\newtheorem{corollary}[theorem]{Corollary}
\newtheorem{question}[theorem]{Question}

\theoremstyle{definition}
\newtheorem{definition}[theorem]{Definition}

\newtheorem{remark}[theorem]{Remark}

\newcommand{\N}{\mathbb{N}} 
\newcommand{\Z}{\mathbb{Z}} 
\newcommand{\C}{\mathbb{C}} 

\newcommand{\F}{\mathscr{F}}

\newcommand{\Lin}{\mathcal{L}}
\newcommand{\p}{\mathcal{P}}
\newcommand{\E}{\mathcal{D}}

\newcommand{\abs}[1]{\left\lvert#1\right\rvert}

\newcommand{\norm}[1]{\lVert#1\rVert}

\begin{document}

\title{Linear dynamics and recurrence properties defined via essential idempotents of $\beta\N$}


\author{Yunied Puig\footnote{Universit\`{a} degli Studi di Milano, Dipartimento di Matematica "Federigo Enriques", Via Saldini 50 - 20133 Milano, Italy. e-mail:puigdedios@gmail.com}}

\date{}
\maketitle

\begin{abstract}
Consider $\mathscr{F}$ a non-empty set of subsets of $\N$.  An operator $T$ on $X$ satisfies property $\p_{\mathscr{F}}$ if for any $U$ non-empty open set in $X$, there exists $x\in X$ such that $\{n\geq 0: T^nx\in U\}\in \mathscr{F}$. Let $\overline{\mathcal{BD}}$ the collection of sets in $\N$ with positive upper Banach density. Our main result is a characterization of sequence of operators satisfying property $\p_{\overline{\mathcal{BD}}}$, for which we have used a deep result of Bergelson and McCutcheon in the vein of Szemer\'{e}di's theorem. It turns out that operators having  property $\p_{\overline{\mathcal{BD}}}$ satisfy a kind of recurrence described in terms of essential idempotents of $\beta \N$. We will also discuss the case of weighted backward shifts. Finally, we obtain a characterization of reiteratively hypercyclic operators.
\end{abstract}

 KEYWORDS: \emph{hypercyclic operator, reiteratively hypercyclic operator,\\ essential idempotent, $\E$-recurrence}\\
 

\section{Introduction}

This paper deals with various recurrence results for linear operators on separable Banach spaces, hence the terminology linear dynamics. These results may be viewed as being roughly analogous to recurrence results for measure-preserving transformations on probability spaces.

 The celebrated theorem of Szemer\'{e}di asserts that any subset of the integers with positive upper density contains arbitrarily large arithmetic progressions. There have been several proofs of Szemer\'{e}di's theorem. One reason for this is that each proof reveals a new facet of Szemer\'{e}di's theorem, allowing to connect different fields of mathematics. For example, Furstenberg's proof in 1977 \cite{Furstenberg} showed that techniques of ergodic theory can be used to prove many Ramsey theoretic results, including Szemer\'{e}di's theorem and certain extensions of Szemer\'{e}di's theorem that were previously unknown. In \cite{Costakis}, Costakis and Parissis examine how some notions of recurrence in linear dynamics are connected to classical notions of recurrence in topological dynamics, in which Szemer\'{e}di's theorem plays an important role. 
 
  In the present work, we improve the results of \cite{Costakis} using a result of Bergelson and McCutcheon \cite{Bergelson0} concerning ergodic Ramsey theory, in the vein of Szemer\'{e}di's theorem. The marriage of combinatorics and recurrence begins to spread to linear dynamics and we hope it will be fruitful for future research. 
\subsection{Preliminaries and main results}
Let $X$ be a complex infinite-dimensional separable Banach space, $T$ a continuous  and linear operator on $X$, denoted $T\in \mathcal{L}(X)$. The main object of study in linear dynamics is the notion of hypercyclicity. The operator $T$ is said to be \emph{hypercyclic} if there exists $x\in X$, called hypercyclic vector, such that its $T$-orbit $\{T^nx: n\geq 0\}$ is dense in $X$.

Birkhoff's transitivity Theorem \cite{Bayart}, asserts that whenever $X$ is a complete, separable and metrizable topological vector space, $T$ is hypercyclic if and only if $T$ is \emph{topologically transitive}, i.e. if for every pair of non-empty open sets $U, V$ in $X$, the set $N(U, V):=\{n\geq 0: T^nU\cap V\neq\emptyset\}$ is non-empty.  Given a set $\F$ of subsets of $\N$, we say that $\F$ is a \emph{family} provided (I.) $ \abs{A}=\infty$ for any $A\in \F$, where $\abs{A}$ denotes the cardinality of $A\subseteq \N$ and (II.) $ A\subset B$ implies $B\in \F$, for any $A\in \F$. According to Birkhoff's transitivity Theorem, it is natural to refine the notion of hypercyclicity in the following way: given a family  $\F$ on $\N$, an operator $T\in \mathcal{L}(X)$ is called an \emph{$\F$-transitive operator} ($\F$-operator for short), if the set $N(U, V)\in \F$, for all non-empty open sets $U, V$ in $X$. This notion was introduced in \cite{Peris1}, which contains an analysis of the hierarchy established between $\F$-operators, whenever $\F$ covers families frequently studied in Ramsey theory.

 An equivalent way of seeing hypercyclicity is the following: an operator $T$ is hypercyclic if  there exists some $x\in X$ such that for every non-empty open set $U\subset X$, the return set $N(x, U):=\{n\geq 0: T^n x\in U\}$ is non-empty. Of particular interest is a strengthened form of hypercyclicity in which for some $x\in X$, its $T$-orbit visits to each non-empty open set  are quantified. Specifically, 
\begin{itemize}
 \item  $T$ is said to be \emph{frequently hypercyclic} if there exists some $x\in X$, called frequently hypercyclic vector, such that for every non-empty open set $U\subset X$, the return set $N(x, U)\in \mathcal{\underline{AD}}$.
 
  \item  $T$ is said to be \emph{$\mathfrak{U}$-frequently hypercyclic} if there exists some $x\in X$, called $\mathfrak{U}$-frequently hypercyclic vector, such that for every non-empty open set $U\subset X$, the return set $N(x, U)\in \mathcal{\overline{AD}}$.
  
   \item  $T$ is said to be \emph{reiteratively hypercyclic} if there exists some $x\in X$, called reiteratively hypercyclic vector, such that for every non-empty open set $U\subset X$, the return set $N(x, U)\in \mathcal{\overline{BD}}$.
\end{itemize}
Here, $\mathcal{\underline{AD}}(\mathcal{\overline{AD}})$ denotes the collection of sets with positive \emph{lower (upper) density}, i.e. $\mathcal{\underline{AD}}=\{A\subseteq \N: \underline{d}(A)>0\} (\mathcal{\overline{AD}}=\{A\subseteq \N: \overline{d}(A)>0\})$ with 
\[
\underline{d}(A)=\liminf_{n\to\infty}\frac{\abs{A\cap [1,n]}}{n} \qquad
\overline{d}(A)=\limsup_{n\to\infty}\frac{\abs{A\cap [1,n]}}{n},
\]
  and $\mathcal{\overline{BD}}$ denotes the collection of sets with positive \emph{upper Banach density}, i.e. $\mathcal{\overline{BD}}=\{A\subseteq \N: \overline{Bd}(A)>0\}$ with $\overline{Bd}(A)=\lim_{s\to\infty}\frac{\alpha^s}{s}$ and 
 \begin{equation}
 \label{def.UpperBanachDens}
 \alpha^s=\limsup_{k\to\infty}\abs{A\cap [k+1, k+s]}.
 \end{equation}

 It is known that 
    \begin{equation}
    \label{eq.density}
     \underline{d}(A)\leq \overline{d}(A) \leq \overline{Bd}(A),
    \end{equation}
    for any $A\subseteq \N$.
Observe that by definition every frequently hypercyclic operator is $\mathfrak{U}$-frequently hypercyclic. For further information about the relationship between these classes of operators, we refer the reader to \cite{Grosse1}, \cite{Bayart} and the references within these, as well as to \cite{BaRu}.  Observe also that every $\mathfrak{U}$-frequently hypercyclic operator is reiteratively hypercyclic. For further information about the relationship between these classes of operators we refer to \cite{Menet} and \cite{BMPP1}.

We are interested in studying a more general notion appearing when the roles of $x$ and $U$ are interchanged in the definition of frequently hypercyclic, $\mathfrak{U}$-frequently hypercyclic and reiteratively hypercyclic operators. Consider a family $\mathcal{A}$ on $\N$. 

\begin{definition}
 We say that a sequence of operators $(T_n)_n$ \emph{satisfies property $\p_{\mathcal{A}}$}, if for any $U$ non-empty open set in $X$, there exists $x\in X$ such that $\{n\in\N: T_nx\in U\}\in \mathcal{A}$.
 
 An operator $T$ satisfies property $\p_{\mathcal{A}}$ if the sequence $(T^n)_n$ satisfies property $\p_{\mathcal{A}}$.
\end{definition}

Obviously, every operator satisfying $\p_{\mathcal{\overline{AD}}}$ satisfies $\p_{\mathcal{\overline{BD}}}$. Concerning the converse we cannot say anything, although we suspect there must be an operator satisfying $\p_{\mathcal{\overline{BD}}}$ and not satisfying $\p_{\mathcal{\overline{AD}}}$, since there exists a reiteratively hypercyclic operator on $c_0(\Z_+)$ which is not $\mathfrak{U}$-frequently hypercyclic \cite{BMPP1}.

     On the other hand, a central notion in topological dynamics is that of recurrence. In our setting, $T\in \mathcal{L}(X)$ is called \emph{recurrent} if for every non-empty open set $U$ in $X$, the set $\{n\geq 0: U\cap T^{-n} U\neq \emptyset\}$ is non-empty. A stronger and well-known notion of recurrence is the following:   an operator $T\in\mathcal{L}(X)$ is \emph{topologically multiply recurrent} if for every non-empty open set $U$ in $X$ and every $r\in \mathbb{N}$, there is some $k\in \N$ such that 
  \begin{equation}
  \label{condition.def.mult.recurrence}
  U\cap T^{-k}U\cap ... \cap T^{-rk}U\neq \emptyset.
  \end{equation}
   Observe that if $T$ satisfies  $\p_{\mathcal{\overline{AD}}}$ then $T$ is topologically multiply recurrent. This follows easily from Szemer\'{e}di's theorem. Indeed, let $U\subset X$ a non-empty open set, then there exists $x\in X$ such that $N(x, U)$ has positive upper density, then by Szemer\'{e}di's theorem it contains arbitrarily long arithmetic progressions. Hence, for any $r\in \N$, there exist $a, k\in \N$ such that $a, a+k, \dots, a+kr\in \{n\geq 0: T^nx\in U\}$. Thus, $T^ax\in \cap_{j=0}^rT^{-jk}U$ and $T$ is topologically multiply recurrent.
  
  Hence, it is natural to wonder whether $T$ is topologically multiply recurrent whenever some family of linear operators $(\lambda_nT^n)_n$ satisfies property $\p_{\mathcal{\overline{AD}}}$ for some sequence of non-zero complex numbers $(\lambda_n)_n$. In fact, this is the content of a result of Costakis and Parissis \cite{Costakis}.
  
 \begin{theorem} (Theorem 3.8 \cite{Costakis})\
 \label{CostPa}
  Let $(\lambda_n)_{n\in\mathbb{N}}$ be a sequence of non-zero complex numbers which satisfies 
  \[
  \lim_{n\to \infty}\frac{|\lambda_n|}{|\lambda_{n+\tau}|}=1
  \]
  for some positive integer $\tau$. If $T\in\mathcal{L}(X)$ is such that the family $(\lambda_nT^n)_n$ satisfies property $\p_{\mathcal{\overline{AD}}}$, then $T$ is topologically multiply recurrent.
  \label{theor.costakis}
  \end{theorem}
 
Theorem 3.8 \cite{Costakis} is in fact a stronger version of what the authors call their main theorem of \cite{Costakis}.

The aim of this work is to generalize Theorem \ref{CostPa}. We do this using a deep result due to Bergelson and Mccutcheon \cite{Bergelson0} concerning ergodic Ramsey theory, in the vein of Szemer\'{e}di's theorem. This leads us to a characterization of a sequence of operators satisfying property $\p_{\overline{\mathcal{BD}}}$. In order to do this, we need to introduce a kind of recurrence stronger than topological multiple recurrence. We first recall some notions from ergodic Ramsey theory.

 Recall that a \emph{filter} is a family that is invariant by finite intersections, i.e. $\F$ is a family such that for any $A\in \F, B\in \F$ implies $A\cap B\in \F$.  The collection of all maximal filters (in the sense of inclusion) is denoted by $\beta \N$. Elements of $\beta \N$ are known as \emph{ultrafilters}; endowed with an appropiate topology, $\beta \N$ becomes the Stone-\v{C}ech Compactification of $\N$. Each point $i\in \N$ is identified with a principal ultrafilter $\mathfrak{U}_i:=\{A\subseteq \N:i\in A\}$ in order to obtain an embedding of $\N$ into $\beta \N$. For any $A\subseteq \N$ and $p\in\beta \N$, the closure of $A$, $cl A$ in $\beta \N$ is defined as follows, $p\in clA$ if and only if $A\in p$. Given $p, q\in\beta \N$ and $A\subseteq \N$, the operation $(\N, +)$ can be extended to $\beta \N$ by defining $A\in p + q$ if and only if $\{n\in \N: -n+A\in q\}\in p$. The operation $+$ on $\beta \N$ is continuous with respect to the topology mentioned above. Thus $\beta \N$ becomes a compact topological semigroup, and according to a famous theorem of Ellis, idempotents (with respect to $+$) exist. Let $E(\N)=\{p\in \beta \N: p=p + p\}$ be the collection of idempotents in $\beta \N$. For further details see \cite{Hindman1}. Given a family $\mathscr{F}$, the \emph{dual} family $\mathscr{F}^*$ consists of all sets $A$ such that $A\cap F\neq \emptyset$, for every $F\in \mathscr{F}$. The following lemma is from \cite{BeDo}:
 \begin{lemma}
 \label{LemmaBeDo}
 $(1)$ If $\F$ is an ultrafilter, then $\F^*=\F$.
 
 $(2)$ If $\F=\cup_\alpha \F_\alpha$, then $\F^*=\cap_\alpha \F^*_\alpha$.
 
 In particular, whenever $\F$ is a union of some collection of ultrafilters, then $\F^*$ is the intersection of the same collection of ultrafilters.
 \end{lemma}
 
   We will be mainly concerned with the so-called \emph{essential idempotents} in $\beta\N$. The collection of essential idempotents is commonly referred to in the literature as $\mathcal{D}$. 
  \begin{definition}
  The collection $\E$ (of $D$-sets) is the union of all idempotents $p\in \beta\N$ such that every member of $p$ has positive upper Banach density. Accordingly, $\E^*$ is the intersection of all such idempotents. 
  \end{definition}
  Now we have all what we need in order to introduce a kind of recurrence stronger than topological multiple recurrence in according with our purposes. We will ask that the intersection in condition (\ref{condition.def.mult.recurrence}) is not only non-empty but also satisfies some condition involving Banach density. While it is well-known that infinitely many $k's$ satisfy condition (\ref{condition.def.mult.recurrence}) for any topologically multiply recurrent operator, this will not be enough for us. In addition we want that this set satisfies very specific algebraic properties, for which we introduce what we call $\E$-recurrence for a linear operator.
  \begin{definition}
 
  An operator $T\in\mathcal{L}(X)$ is \emph{topologically $\E$-recurrent with respect to $\lambda=(\lambda_n)_n$} if there exists some $p\in\E$ such that for any non-empty open set $U$ in $X$, there exists $x\in X$, such that for any $r\in \N$, we have
 \[
 \Big\{k\in\N: \overline{Bd}\Big(a\in\N: \lambda_aT^ax\in\cap_{j=0}^{r} T^{-jk}(U)\Big)>0 \Big\}\in p.
 \]
 \end{definition}
 
  In the case, $(\lambda_n)_n=1$, we simply say $T$ topologically $\E$-recurrent. Note that in particular, topological $\E$-recurrence implies  topological multiple recurrence. However the converse is not true.
 
 \begin{proposition}
 \label{top.multrec.butnot.Erec}
 Let $X=c_0(\Z_+)$ or $l^p(\Z_+), 1\leq p<\infty$. Then there exists a topologically multiply recurrent operator on $X$ which is not topologically $\E$-recurrent.
 \end{proposition}

 In order to state our main result, we need the notion of limit along a collection of sets.
 
 \begin{definition} (Limit along a filter)\ \
 
  Given a filter $\F$ on $\N$,  $\F-\lim_n T^n(x):=y$ if and only if for every open neighbourhood $V$ of $y$, the set $\{n\geq 0: T^n(x)\in V\}\in \mathscr{F}$.
 \label{def.F-limit}
 \end{definition}
 
 Denote $\mathcal{\underline{BD}}_1$ the collection of sets with lower Banach density 1, where the \emph{lower Banach density} of a set $A\subseteq \N$ is defined in the same way as the upper Banach density taking limit inferior instead of limit superior in condition (\ref{def.UpperBanachDens}). We remark that $\mathcal{\underline{BD}}_1$ is a filter, this is well-known and discussed in detail in the proof of  Theorem \ref{theor.recurrence} below. We state our main result:
  \begin{theorem}
   \label{theor.recurrence} 
     Let $(\lambda_n)_n$ be a sequence of non-zero complex numbers and let $p\in\E$ such that there exists $A\in p$ for which
     \begin{equation}
    \mathcal{\underline{BD}}_1-\lim_n\Big|\frac{\lambda_n}{\lambda_{n+k}}\Big|=1, ~ \forall k\in A.
     \label{equat3'}
     \end{equation}
   Then the family $(\lambda_nT^n)_n$ acting on  $X$ satisfies the property $\p_{\mathcal{\overline{BD}}}$ if and only if  $T$ is topologically $\E$-recurrent with respect to $(\lambda_n)_n$.
    \end{theorem}

 Thus we have the following diagram, where each solid arrow indicates implication, and the dashed arrow indicates that implication fails to hold.

 
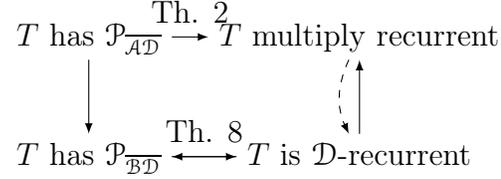
\begin{center}
\begin{tikzpicture}[>=latex,text height=1.5ex,text depth=0.25ex]
  \matrix[row sep=0.5cm,column sep=0.5cm] {
    \node (u1) {$T$ has $\p_{\mathcal{\overline{AD}}}$}; &
    
        \node (u2) {$T$ multiply recurrent}; &
         &
    \\
        \\
    \node (u3) {$T$ has $\p_{\mathcal{\overline{BD}}}$}; &
        \node (u4) {$T$ is $\E$-recurrent}; &
            \\
        };
  
    \path[->]
    
   (u1) edge node [above]  {Th. \ref{CostPa}} (u2)
        (u1) edge (u3)  
    (u4) edge (u2)
    (u3) edge (u4)
    (u4) edge node [above]  {Th. \ref{theor.recurrence}} (u3)      
   (u2) edge [bend right=25, dashed]  (u4)

   ;                      
 \end{tikzpicture}    
 \end{center}

 First, observe from this diagram that the converse of Theorem \ref{CostPa} can not be obtained, i.e. there exists a multiply recurrent operator which does not have $\p_{\mathcal{\overline{AD}}}$. On the other hand, Theorem \ref{theor.recurrence} is in fact a generalization of Theorem \ref{CostPa}, as can be deduced from the diagram.
 
    Finally, we obtain a characterization of reiteratively hypercyclic operators, which gives us much more information about the behavior of the return sets  of reiteratively hypercyclic operators.
  \begin{theorem}
  \label{charact.rhc.Op}
  An operator $T\in \Lin (X)$ is reiteratively hypercyclic if and only if there exists some $x\in X$ such that for any non-empty open set $U$ in $X$ and any $r\in \N$, it holds
  \[
  \Big\{k\in \N: \overline{Bd}\Big(a\in \N: T^ax\in T^{-k}U\cap \cdots\cap T^{-rk}U\cap U\Big)>0\Big\}\in \E^*.
  \]
  \end{theorem}

\section{Operators satisfying property $\p_{\mathcal{\overline{BD}}}$}

\subsection{Proof of theorem \ref{theor.recurrence}}
   
   Suppose $T$ is topologically $\E$-recurrent with respect to $(\lambda_n)_n$, then there exists some $p\in \E$ such that for any non-empty open set $U$ in $X$, there exists $x\in X$ such that $\{k\in\N: \overline{Bd}(a\in\N: \lambda_aT^ax\in U\cap T^{-k}U)>0 \}\in p$. In particular, we can pick some $k$ such that $\overline{Bd}(a\in\N: \lambda_aT^ax\in U\cap T^{-k}U)>0$. Hence, $\overline{Bd}(a\in\N: \lambda_aT^ax\in U)>0$ and $(\lambda_nT^n)_n$ has property $\p_{\mathcal{\overline{BD}}}$.

 Conversely, let $(\lambda_n)_n$ a sequence of non-zero complex numbers as in the statement of the theorem, then the family $(\lambda_nT^n)_n$ has the property $\p_{\mathcal{\overline{BD}}}$ if and only if $(|\lambda_n|T^n)_n$ has the property $\p_{\mathcal{\overline{BD}}}$, and the proof follows the same lines as Lemma 3.7 \cite{Costakis}, replacing $\overline{Bd}$ instead of $\overline{d}$. So, we may assume that $(\lambda_n)_n$ is a sequence of positive numbers such that for some $p\in\E$ and some $A\in p$, it is satisfied $\lim_n \lambda_n / \lambda_{n+k}=1$ for every $k\in A$.
 
 Let $U$ a non-empty open set in $X$ and $r\in \mathbb{N}$, then there exists $y\in U$ and a positive number $\epsilon$ such that $B(y; \epsilon)\subseteq U$, where $B(y; \epsilon)$ denote the open ball centered at $y$ with radius $\epsilon$. Hence by the property $\mathcal{P}_{\mathcal{\overline{BD}}}$ of the family $(\lambda_nT^n)_n$,  there exists $x\in X$ such that 
  \[
  F=\{n\in \mathbb{N}:\lambda_n T^n x\in B(y, \epsilon /2)\}
  \]
   has positive upper Banach density. Consider the family of polynomials $g_1(k)=-k,...,g_r(k)=-rk$, hence $g_1,...g_r\in \mathcal{G}_a$ the group of admissible generalized polynomials, see page 10 \cite{Bergelson0} for the definition. Now by Theorem 1.25 \cite{Bergelson0} we have that
 
 
  \begin{equation}
 W:= \Big\{k\in\mathbb{N}: \overline{Bd}\Big(F\cap (F-k)\cap...\cap (F-rk)\Big)>0\Big\}\in \E^*.
 \label{equat4}
  \end{equation}
 Hence, $W\in q, \forall q \in\E$.
 
 On the other hand, recall that (\ref{equat3'}) holds for some $p\in\E$ and some $A\in p$,  hence,

 
 \begin{equation}
 W\cap A\in p.
  \label{equat5}
  \end{equation}
  
  Fix $k\in W\cap A$. Denote
  \[
  M_{k, r}:=F\cap (F-k)\cap...\cap (F-rk)=\{a: a, a+k,...,a+rk \in F\}.
  \]

 As in the proof of Theorem 3.8 \cite{Costakis}, we agree the following notation 
 \begin{equation}
 \label{eq.def.u}
 u:=\lambda_a T^a x\in B(y; \epsilon/2)
 \end{equation}
 \[
  u_j:=\lambda_{a+jk}T^{a+jk} x= \frac {\lambda_{a+jk}}{\lambda _a} T^{jk} u\in B(y; \epsilon/2)
 \]
  for every $a\in M_{k, r}$ and $j=1\dots r$. Let $M>0$ such that $\norm{u_j}\leq M$ for any $j=1\dots r$.
 
  On the other hand, 
   \[
    || T^{jk} u-u_j||=\Big\|  \frac{\lambda _a}{\lambda_{a+jk}} u_j - u_j\Big\|= \Big | \frac{\lambda_{a}}{\lambda_{a+jk}}-1\Big | \|u_j\|
   \]
    Now, set 
   \[
   I_{j,k}=\Big\{a\in\N: \Big|\frac{\lambda_a}{\lambda_{a+jk}}-1\Big|<\epsilon/2M\Big\}.
   \]
    It is a well-known fact and easy to verify that $\overline{\mathcal{BD}}^*=\underline{\mathcal{BD}}_1$. Hence, 
   \begin{equation}
   I_{j,k}\in\mathcal{\overline{BD}}^*, 
   \label{equat6}
   \end{equation}
   for every $j=1\dots r$, because by hypothesis $k\in W\cap A$ implies 
   
\[
   \mathcal{\overline{BD}}^*-\lim_{n}\frac{\lambda_n}{\lambda_{n+k}}=1,
   \]
  and at the same time, the following is true.
  
    \textbf{Fact 1:}
    $\mathcal{\overline{BD}}^*-\lim_{n}\frac{\lambda_n}{\lambda_{n+k}}=1 $ implies $\mathcal{\overline{BD}}^*-\lim_{n}\frac{\lambda_n}{\lambda_{n+jk}}=1$,
 
   for any $j=1\dots r$.
   
   In fact, 
   \[
   \lambda_a/\lambda_{a+jk}= \lambda_a/\lambda_{a+k} \cdot \lambda_{a+k}/\lambda_{a+2k}\cdot...\cdot \lambda_{a+(j-1)k}/\lambda_{a+jk}.
   \]
    Now, let $V$ an open neighbourhood of $1$, then one can find $V_1, V_2, \dots , V_j$ neighbourhoods of $1$ such that 
   \[
   \Big\{n\geq 0: \frac{\lambda_n}{\lambda_{n+k}}\in V_1\Big\} \cap \Big\{n\geq 0: \frac{\lambda_{n+k}}{\lambda_{n+2k}}\in V_2\Big\}\cap...\cap \Big\{n\geq 0: \frac{\lambda_{n+(j-1)k}}{\lambda_{n+jk}}\in V_j\Big\}\subseteq
   \]   
   
   \[
   \subseteq \Big\{n\geq 0: \frac{\lambda_n}{\lambda_{n+k}}\cdot \frac{\lambda_{n+k}}{\lambda_{n+2k}}\cdot ...\cdot \frac{\lambda_{n+(j-1)k}}{\lambda_{n+jk}}\in V\Big\}=\Big\{n\geq 0: \frac{\lambda_n}{\lambda_{n+jk}}\in V\Big\}.
   \]
   
   By hypothesis, 
   \[
   \Big\{n\geq 0: \lambda_{n+(t-1)k} /\lambda_{n+tk}\in V_t\Big\}\in \mathcal{\overline{BD}}^*,
   \]
    for $t=1\dots j$. Hence because $\mathcal{\overline{BD}}^*$ is a filter (see Fact 2 below), we have 
   \[
   \bigcap_{t=1}^j\Big\{n\geq 0: \frac{\lambda_{n+(t-1)k}}{\lambda_{n+tk}}\in V_t\Big\}\in \mathcal{\overline{BD}}^*.
   \]
    Consequently, $\Big\{n\geq 0: \lambda_n/\lambda_{n+jk}\in V\Big\}\in \mathcal{\overline{BD}}^*$. Finally due to the arbitrariness of $V$, we conclude 
    \[
    \Big\{n\geq 0: \lambda_n/\lambda_{n+jk}\in V\Big\}\in \mathcal{\overline{BD}}^*,
    \]
     for every $V$ open neighbourhood of $1$ and any $j=1\dots r$, which concludes the proof of Fact 1.

    \textbf{Fact 2:}
   The family $\mathcal{\overline{BD}}^*$ is a filter. In fact, it can be written as intersection of ultrafilters, i.e.
   \[
\mathcal{\overline{BD}}^*=\bigcap_{p\in \mathscr{D}}p^*=\bigcap_{p\in \mathscr{D}}p=\{A\subseteq \N: A\in p, \forall p \in \mathscr{D}\},
   \]
   where $\mathscr{D}=\{p\in \beta \N: \overline{Bd}(A)>0, \forall A\in p\}$. 
   
   First, by Lemma \ref{LemmaBeDo} if $\mathscr{F}$ is an ultrafilter, then $\mathscr{F}=\mathscr{F^*}$. Second, obviously
   \[
   \bigcup_{p\in \mathscr{D}}p=\{A\subseteq \N: A\in p, p \in \mathscr{D}\}\subseteq \mathcal{\overline{BD}}.
   \]
    Conversely, Let $A\in \mathcal{\overline{BD}}$, then by Lemma 2.3 \cite{Hindman} there exists $p\in \mathscr{D}$ such that $A\in p$, then $\mathcal{\overline{BD}}=\bigcup_{p\in \mathscr{D}}p$ and consequently $\mathcal{\overline{BD}}^*=\bigcap_{p\in \mathscr{D}}p^*$, which concludes the proof of Fact 2.
    
    Hence by (\ref{equat6}), we have that $I_k:=\cap_{j=1}^rI_{j,k}\in\mathcal{\overline{BD}}^*=\mathscr{D}^*$, i.e.
     \begin{equation}
   I_k\in q , \forall q\in \mathscr{D}.
   \label{equat8}
   \end{equation}
   
   Now, $\overline{Bd}(M_{k, r})>0$ by (\ref{equat4}). Hence by Lemma 2.3 \cite{Hindman} there exists $\widetilde{p}\in\mathscr{D}$ such that 
  \begin{equation}
   M_{k, r}\in \widetilde{p}.
   \label{equat9}
   \end{equation}
   
  By (\ref{equat8}) and (\ref{equat9}) it results $A_{k, r}:=I_k\cap M_{k, r}\in\widetilde{p}$. Hence, 
      \begin{equation}
   \overline{Bd}(A_{k, r})>0
   \label{equat10}
   \end{equation}
   
   and 
   \[
   A_{k, r}\subseteq \Big\{a\in \N: \norm{T^{jk}u-u_j}<\epsilon/2, \norm{u_j-y}<\epsilon/2,  j=1\dots r\Big\},
   \]
   then
   \begin{equation}
   \label{eq.Akr}
    A_{k, r}\subseteq \Big\{a\in \N: \norm{T^{jk}u-y}<\epsilon, j=1\dots r\Big\}.
   \end{equation}

   Hence  by (\ref{eq.def.u}) and (\ref{eq.Akr})  we obtain
   \[
   u, T^k(u), ..., T^{rk}(u)\in U,
   \]
   for every $a\in A_{k, r}$.

   Now, by (\ref{equat5}) and (\ref{equat10}) we have,
   \[
   \Big\{k\in\N: \overline{Bd}\Big(a\in\N: \lambda_aT^ax\in\cap_{j=0}^r T^{-jk}(U)\Big)>0\Big\}\in p
   \]
   
   and $T$ is topologically $\E$-recurrent with respect to $\lambda=(\lambda_n)_n$.     
     \begin{remark}
     $(a)$ Observe that in the sufficiency part of the proof of Theorem \ref{theor.recurrence}, condition (\ref{equat3'}) is superfluous.
     
     $(b)$ In Example 3.13 \cite{Costakis}, Costakis and Parissis showed that the hypothesis $\lim_{n\to \infty} \abs{\lambda_n}/\abs{\lambda_{n+\tau}}=1$, for some positive integer $\tau$, in Theorem \ref{CostPa}, cannot be replaced by the hypothesis $\lim_{n\to\infty, n\in A} \abs{\lambda_n}/\abs{\lambda_{n+\tau}}=1$ for some positive integer $\tau$ and $A\subset \N$ with $\underline{Bd}(A)=1$. We remark that there exists a linear operator $T$ on $c_0(\Z_+)$ such that for any finite set $A\subset \N$, there exists a sequence of complex numbers $(\lambda_n)_n$ such that $\underline{Bd}_1$-$\lim_n \abs{\lambda_n}/\abs{\lambda_{n+k}}=1$, for every $k\in A, (\lambda_nT^n)_n$ satisfies property $\p_{\mathcal{\overline{BD}}}$ and $T$ is not recurrent. We can use Example 3.13 \cite{Costakis} of Costakis and Parissis to show that the hypothesis of Theorem \ref{theor.recurrence} cannot be weakened in this sense. We give the details for the sake of completeness. Let $B$ the unilateral backward shift on $c_0(\Z_+)$, defined by: $B e_n:=e_{n-1}, n\geq 1$,  $Be_0:=0$, where $(e_n)_{n\in \Z_+}$ denotes the canonical basis of $c_0(\Z_+)$. Suppose $A\subset \N$ is a finite set and $M=\max_{a\in A}a$. Define $\lambda_n=2^{2^r}$ if $n\in [2^{r-1}, 2^r-M]$ and $S=\cup_{r=M+1}^\infty [2^{r-1}, 2^r-2M]$. Note that $S\subseteq \{n\in \N: \abs{\lambda_n}/\abs{\lambda_{n+k}}=1, \forall k\in A\}$ and $\underline{Bd}(S)=1$. Hence, 
     \[
     \underline{Bd}_1-lim_n \abs{\frac{\lambda_n}{\lambda_{n+k}}}=1, ~ \forall k\in A.
     \]
     Using the frequent universality criterion from \cite{BoGr1} and \cite{BoGr2} it is not difficult to see that $(\lambda_nB^n)_n$ is frequently universal, i.e. there exists $x\in X$ such that for every non-empty open set $U\subset X$, the set $\{n\in \N: \lambda_nB^nx\in U\}$ has positive lower density. In particular, $(\lambda_nB^n)_n$ has property $\p_{\mathcal{\overline{BD}}}$ and $B$ is not recurrent.
     
     However, we cannot say anything if we replace the hypothesis 
     \[
     \underline{Bd}_1-\lim_n \abs{\lambda_n}/\abs{\lambda_{n+k}}=1, ~ \forall k\in A,
     \]
      for some $A\in p$ and $p\in \E$ in Theorem \ref{theor.recurrence}, by $\underline{Bd}_1$-$\lim_n \abs{\lambda_n}/\abs{\lambda_{n+k}}=1$, for every $k\in A$ with $\underline{Bd}(A)=0$.
     \end{remark}
        
     \subsection{Adjoint of multiplication operators}
      An easy application of Theorem \ref{theor.recurrence} can be seen in the frame of adjoint of multiplication operators (see \cite{Grosse1}, \cite{Bayart}) for an introduction. 
          
    Fix a non-empty open connected set $\Omega\subset \C^n, n\in \N$, and $H$ a Hilbert space of holomorphic functions such that $H\neq \{0\}$ and  for every $z\in \Omega$, the point evaluation functionals $f\mapsto f(z), f\in H$, are bounded.
    
     Recall that every complex valued function $\phi: \Omega \to \C$ such that the pointwise product $\phi f\in H$, for every $f\in H$ is called a multiplier of $H$, and defines a multiplication operator $M_\phi: H\to H$ defined as 
     \[
     M_\phi(f)=\phi f, \quad f\in H.
     \]
     
     The following is an improvement of Proposition 6.1 \cite{Costakis}.
\begin{corollary}
\label{adjMulOp}
Suppose that every non-constant bounded holomorphic function $\phi$ on $\Omega$ is a multiplier of $H$ such that $\norm{M_\phi}=\norm{\phi}_\infty$. Then for each such $\phi$ the following are equivalent.

i) $M_\phi^*$ is topologically $\E$-recurrent

ii) $M_\phi^*$ is recurrent

iii) $M_\phi^*$ is frequently hypercyclic

iv) $M_\phi^*$ is hypercyclic

v) $\phi (\Omega)\cap \mathbb{T}\neq \emptyset$.
\end{corollary}
\begin{proof}
By Proposition 6.1 \cite{Costakis}, conditions $ii)-v)$ are equivalent. Condition $iii)$ implies $i)$. Indeed, suppose $M_\phi^*$ is frequently hypercyclic, then obviously $M_\phi^*$ is reiteratively hypercyclic, hence topologically $\E$-recurrent by Theorem \ref{theor.recurrence}. Finally, $i)$ implies $ii)$ since every topologically $\E$-recurrent operator is clearly recurrent. 
\end{proof}
     
    \subsection{Weighted shifts satisfying property $\p_{\mathcal{\overline{BD}}}$}
    In this section we will prove Proposition \ref{top.multrec.butnot.Erec} and show another consequence of Theorem \ref{theor.recurrence}. 
    
    An important class of operators in linear dynamics is the weighted shifts.
    \begin{definition}(weighted backward shifts)

Each bilateral bounded weight $w=(w_k)_{k\in \Z}$, induces a \emph{bilateral weighted backward shift} $B_w$ on $X=c_0(\Z)$ or $l^p(\Z)$, given by $B_{w}e_k:=w_{k}e_{k-1}$, where $(e_k)_{k\in \Z}$ denotes the canonical basis of $X$.

   Similarly, each unilateral bounded weight $w=(w_n)_{n\in \Z_+}$ induces a \emph{unilateral weighted backward shift} $B_w$ on $X=c_0(\Z_+)$ or $l^p(\Z_+)$, given by $B_{w}e_n:=w_{n}e_{n-1}, n\geq 1$ with $B_{w}e_0:=0$, where $(e_n)_{n\in \Z_+}$ denotes the canonical basis of $X$.
\end{definition}

 \subsubsection{A multiply recurrent operator which is not $\E$-recurrent}
 
 Now we proceed to prove Proposition \ref{top.multrec.butnot.Erec} but first we need to point out the following:
  \begin{proposition}
   Let $B_w$ a weighted backward shift on $X=c_0(\Z_+)$ or $l^p(\Z_+), 1\leq p<\infty$. The following are equivalent:
  
  i) $B_w$ topologically multiply recurrent
  
  ii)   $\forall M>0, \exists m_0: \forall m>m_0, \exists n: \min_{1\leq l\leq m} \{|w_1w_2...w_{ln}|\}>M$.
  \label{proposition.remark}
  \end{proposition}
\begin{proof}
It is easy to see that condition $ii)$ is equivalent to
  \begin{equation}
  \label{condition.top.mult.rec}
 \forall m\in \N, \ \  \sup_{n\in\N}\Big\{\min_{1\leq l \leq m}\{|w_1w_2...w_{ln}|\}\Big\}=\infty.
  \end{equation}
  By Proposition 4.3 \cite{BP} condition (\ref{condition.top.mult.rec}) is equivalent to say $B_w\oplus B_w^2\oplus \cdots\oplus B_w^m$ is hypercyclic on $X^m$, for every $m\in\N$. Finally, Proposition 5.3 \cite{Costakis} is enunciated for bilateral weighted shifts, but obvious modifications asserts that the operator $B_w\oplus B_w^2 \oplus \cdots\oplus B_w^m$ is hypercyclic on $X^m$, for every $m\in \N$ if and only if $B_w$ is topologically multiply recurrent.
\end{proof}
    
  Recall that a set $A\subseteq \N$ is \emph{syndetic} if it has bounded gaps, i.e. there exists $k\in \N$ such that for any $n\in \N, \{n+i: 1\leq i \leq k\}\cap A\neq \emptyset$.  
  
  In order to prove Proposition \ref{top.multrec.butnot.Erec} we will need the following characterization of syndetic unilateral weighted backward shifts \cite{Peris1}.
  \begin{proposition} \cite{Peris1}
   \label{charact.Bw.syndetic}
   
  Let $X=c_0(\Z_+)$ or $l^p(\Z_+)$ and $w=(w_n)_{n\in \Z_+}$ a unilateral bounded weight, then the following are equivalent:
  
  i) $B_w$ is syndetic operator
  
  ii) the set $\{n\in \N:\prod_{i=1}^n\abs{w_i}>M\}$ is syndetic, for any $M>0$.
  \end{proposition}
  \begin{remark}
  \label{remark.propPBDImpliesSyndetic}
  Observe that if $T$ is a hypercyclic operator that satisfies property $\p_{\mathcal{\overline{BD}}}$ then it is a syndetic operator. In fact, Let $U, V$ non-empty open sets in $X$. Pick, $n\in N(U, V)$, then $T^{n}(U)\cap V\neq\emptyset$. Denote the non-empty open set $U_n:=U\cap T^{-n}(V)$ and pick $x\in X$ such that $\overline{Bd}\left(N(x, U_n)\right)>0$.
 
 On the other hand, it is a well known fact that 
 \begin{equation}
 N(x, U_n)-N(x, U_n)+n\subseteq N(U, V).
  \label{ident.grivaux1}
 \end{equation}
  Let $s_1, s_2\in N(x, U_n)$. Then, we have
\[
 T^{s_1-s_2+n}(T^{s_2}x)=T^{n}(T^{s_1}x)\in V.
\]
and identity (\ref{ident.grivaux1}) holds. A theorem of Erd\"{o}s and S\'{a}rk\"{o}zy \cite{ST} asserts that $A-A$ is a syndetic set whenever $A$ has positive upper Banach density. Since $B+n$ is a syndetic set as well as $B$ is syndetic, we conclude that $N(x, U_n)-N(x, U_n)+n$ is a syndetic set, hence $N(U, V)$ is a syndetic set. By definition, $T$ is a syndetic operator.
  \end{remark}

\textbf{Proof of Proposition \ref{top.multrec.butnot.Erec}}
 
  Note that for weighted backward shifts, be hypercyclic is equivalent to be recurrent \cite{Costakis}. Hence by Theorem \ref{theor.recurrence}, it suffices to find $B_w$ non syndetic operator in virtue of Remark \ref{remark.propPBDImpliesSyndetic} and satisfying condition ii) of Proposition \ref{proposition.remark}.
  
   Let us construct $(b_n)_{n\in \N}=B\subseteq \N$ with the property
   \[
   \forall m\in \N, \  \exists n\in \N : ln\in B, \ \ \forall 1\leq l \leq m
   \]
   and define a weight $w=(w_n)_n$ in such a way that $A_1=\{n\in \N: \prod_{i=1}^n|w_i|>1\}$ be non-syndetic and $w$ satisfies condition ii) of Proposition \ref{proposition.remark} on $B$.  Denote $w^n=(w_1, ..., w_n)$. For better understanding we set $w_n^*$ for indicate that $n\in A_1$.
   
   Let $m=1, n=1$ and define $b_1=1\cdot 1=1$. Then, $w^3=(2^*, 1/2^2, 2)$. Let $m=2$, take $n=4$ and define $b_2=1\cdot 4=4, b_3=2\cdot 4=8$. Then, introducing an increasing gap on $A_1$ we set 
   \[
   w^{b_3+3}=(\underbrace{2^*}_{w_{b_1}}, \frac{1}{2^2}, 2, \underbrace{2^*}_{w_{b_2}}, 2^*, 2^*, 2^*, \underbrace{2^*}_{w_{b_3}}, 1/2^7, 2, 2).
   \]
   Now, in order to satisfy condition ii) of Proposition \ref{proposition.remark} we must define $b_4$ at least equal to $b_3+4+(b_3-b_2+1)=17$. Hence, for $m=3$, take $n=17$ and define $b_4=1\cdot 17=17, b_5=2\cdot 17=34, b_6=3\cdot 17=51$, introducing the corresponding increasing gap on $A_1$ we have
   \[
   w^{b_6+4}=(\underbrace{2^*}_{w_{b_1}}, \frac{1}{2^2}, 2, \underbrace{2^*}_{w_{b_2}},..., \underbrace{2^*}_{w_{b_3}}, 1/2^7, 2, 2, 2^*, ..., \underbrace{2^*}_{w_{b_4}}, ..., \underbrace{2^*}_{w_{b_5}},..., \underbrace{2^*}_{w_{b_6}}, 
   \]
   \[1/2^{43}, 2, 2, 2).
   \]
   Again in order to satisfy condition ii) of Proposition \ref{proposition.remark} we must define $b_7$ at least equal to $b_6+5+(b_6-(b_3+4)+1)=97$. Hence for $m=4$, take $n=97$ and define $b_7=1\cdot 97=97, b_8=2\cdot 97, b_9=3\cdot 97, b_{10}=4\cdot 97=388$, introducing the corresponding increasing gap on $A_1$ we have
 \[
 w^{b_{10}+5}=(\underbrace{2^*}_{w_{b_1}}, \frac{1}{2^2}, 2, \underbrace{2^*}_{w_{b_2}},..., \underbrace{2^*}_{w_{b_3}}, 1/2^7, 2, 2, 2^*, ..., \underbrace{2^*}_{w_{b_4}}, ..., \underbrace{2^*}_{w_{b_5}},..., \underbrace{2^*}_{w_{b_6}}, 1/2^{43}, 
 \]
 \[
  2, 2, 2, 2^*, ..., \underbrace{2^*}_{w_{b_7}},..., \underbrace{2^*}_{w_{b_8}}, ... \underbrace{2^*}_{w_{b_9}}, \underbrace{2^*}_{w_{b_{10}}}, 1/2^{337}, 2, 2, 2, 2).
\]
 an so on.
    Clearly $B_w$ satisfies condition ii) of Proposition \ref{proposition.remark} and by Proposition \ref{charact.Bw.syndetic} is not a syndetic operator, and this concludes the proof of Proposition \ref{top.multrec.butnot.Erec}.

\subsubsection{Weighted shifts satisfying property $\p_{\mathcal{\overline{BD}}}$}
In general, an operator satisfying property $\p_{\mathcal{\overline{BD}}}$ is not necessarily hypercyclic, consider for example the identity operator. But in the context of weighted backward shifts on $X=l^p$ or $c_0$, operators satisfying property $\p_{\mathcal{\overline{BD}}}$ are necessarily hypercyclic, even more, we show that any weighted backward shift $B_w$ having property $\p_{\mathcal{\overline{BD}}}$ satisfies a stronger condition, i.e. $B_w\oplus B_w^2\oplus ... \oplus B_w^r$ is $\E^*$-operator on $X^r$, for any $r\in \N$. 
 
 \begin{proposition}
  \label{bilateral.B_w.top.Erecurrent}
Let $w=(w_k)_{k\in \Z}\big(w=(w_n)_{n\in \Z_+}\big)$ be a bounded weight sequence and $B_w$ a bilateral (unilateral) weighted backward shift on $X=l^p(\Z)$ or $c_0(\Z)\big(X=l^p(\Z_+)$ or $c_0(\Z_+)\big)$. If $B_w$ satisfies property $\p_{\mathcal{\overline{BD}}}$  then $B_w\oplus B_w^2\oplus ... \oplus B_w^r$ is $\E^*$-operator on $X^r$, for any $r\in \N$. 
 \end{proposition}
  
  Notice that by Lemma \ref{LemmaBeDo}, $\E^*$ is a filter, since it can be written as an intersection of ultrafilters, indeed $\E^*=\cap_{p\in \E} p$.

Let $M>0, j\in \Z$ and $w=(w_i)_i$ a bounded weight. Let us denote
    \[
    A_{M; j}=\Big\{n\in \N: \prod_{i=j+1}^ {j+n} \abs{w_i}>M\Big\}, \quad   \bar{A}_{M; j}=\Big\{n\in \N: \frac{1}{\prod_{i=j-n+1}^j \abs{w_i}}>M\Big\}.
    \]

    In order to prove Proposition \ref{bilateral.B_w.top.Erecurrent}, we will need the following characterization of $\F$-operators for weighted shifts and $\F$ a filter, see \cite{Peris1}.
    \begin{proposition} \cite{Peris1}
    \label{charact.Bw.filter.op}
    
    Let $\F$ a filter on $\N$, then
    
    i) the bilateral weighted backward shift $B_w$ is $\F$-operator on $X=c_0(\Z)$ or $l^p(\Z)$ if and only if $A_{M;j}\in \F$ and $\bar{A}_{M; j}\in \F$, for any $M>0, j\in \Z$
    
    ii) the unilateral weighted backward shift $B_w$ is $\F$-operator on $X=c_0(\Z_+)$ or $l^p(\Z_+)$ if and only if $A_{M;j}\in \F$, for any $M>0, j\in \Z_+$.
    \end{proposition}
    The following Lemma is an easy consequence of Proposition \ref{charact.Bw.filter.op}.
 \begin{lemma} \cite{Peris1}
 \label{corol.weighted.shifts.F.oper}
 
 Let $\mathscr{F}$ a filter, $r\in \N$ and $B_w$ a bilateral weighted backward shift on $X=l^p(\Z)$ or $c_0(\Z)$, then the following are equivalent:
 
  i) $A_{M, j}\in l\F$ and $\bar{A}_{M, j}\in l\F$, for any $1\leq l\leq r, M>0, j\in \Z$
 
 ii) $B_w\oplus B_w^2\oplus ... \oplus B_w^r$ is $\mathscr{F}$-operator on $X^r$.
 \end{lemma}

 \begin{proof}
 By Proposition \ref{charact.Bw.filter.op}, we have that $A_{M, j}\in l\F$ and $\bar{A}_{M, j}\in l\F$, for any $M>0, j\in \Z$ and any $1\leq l \leq r$ is equivalent to $N_{B_w^l}(U, V)\in\mathscr{F}$ for any $U, V$ non-empty open sets in $X$ and $1\leq l \leq r$. On the other hand, condition $ii)$ is equivalent to $\cap_{l=1}^r N_{B^l_w}(U_l, V_l)\in \F$, for any pair of finite collection of non-empty open sets $(U_l, V_l)_{l=1}^r$. The conclusion follows since $\F$ is a filter.
 \end{proof}

\textbf{Proof of Proposition \ref{bilateral.B_w.top.Erecurrent}.}
 
    \textbf{Fact 3:} If $B_w$ is topologically $\E$-recurrent then, for any $r\in \N$ we have that $A_{M; j}\in l\E^*$ and $\bar{A}_{M; j}\in l\E^*$, for any $j\in \Z, M>0$ and $1\leq l\leq r$.
    
    Let $M>0, j\in \Z$. We must show 
   \[
   \forall r\in\N, \ \ \exists W\in\E^*: lk\in A_{M, j}, lk\in \bar{A}_{M, j}\ \ \forall k\in W, \ \ 1\leq l\leq r.
   \]
   
   Let $r\in \N$. Pick $\delta>0$ such that $(1-\delta)/\delta >M$. Consider the open ball $B(e_j, \delta)=\{x\in X: \norm{x-e_j}<\delta\}$. $B_w$ topologically $\E$-recurrent implies there exists $W\in \E^*$ such that for each $k\in W$ there exists
   \begin{equation}
   y\in B(e_j, \delta)
   \label{equation.bol1}
   \end{equation}
   such that 
   \begin{equation}
   T^{lk}y\in B(e_j, \delta)
   \label{equation.bol2}
   \end{equation}

  for any $1\leq l \leq r$. The existence of $W\in \E^*$ is due to the fact that we are considering $(\lambda_n)_n=1$ and hence in (\ref{equat5}), $A$ can be taken as $\N$.
  
  By (\ref{equation.bol1}),
  \begin{equation}
  \abs{y_j-1}<\delta , \qquad \abs{y_t}<\delta \ \mbox{for} \ t\neq j.
  \label{equation.bol3}
  \end{equation}
  By (\ref{equation.bol2}),
  \begin{equation}
  \prod_{i=1}^{lk}\abs{w_{i+j}y_{j+lk}-1}<\delta, \qquad \prod_{i=1}^{lk}\abs{w_{i+t}y_{t+lk}}<\delta \ \mbox{for} \ t\neq j.
  \label{equation.bol4}
  \end{equation}
for any $1\leq l\leq r$.

 Now by (\ref{equation.bol4}), we have 
  \begin{equation}
  \Big|\prod_{i=1}^{lk}w_{i+j}y_{j+lk}\Big|>1-\delta.
  \label{equation.bol5}
  \end{equation}
Thus by (\ref{equation.bol3}) and (\ref{equation.bol5}),
 \[
  \prod_{i=1}^{lk}\abs{w_{i+j}}>\frac{1-\delta}{\delta} >M,
 \]
and  $lk\in A_{M, j}$, for any $1\leq l\leq r$.
 
 On the other hand, by \eqref{equation.bol4},
 \[
 \prod_{i=j-lk+1}^j \abs{w_i y_j}<\delta<\frac{1-\delta}{M}. 
 \]
 Furthermore, by \eqref{equation.bol3} we get
 \[
 \prod_{i=j-lk+1}^j \abs{w_i}(1-\delta)< \prod_{i=j-lk+1}^j \abs{w_i y_j}<\frac{1-\delta}{M}.
  \]
  Hence, $\prod_{i=j-lk+1}^j\abs{w_i}<1/M$ and $lk\in \bar{A}_{M, j}$, for any $1\leq l\leq r$. This concludes the proof of Fact 3.
 
   By Fact 3 and Lemma \ref{corol.weighted.shifts.F.oper} we have that, if $B_w$ is topologically $\E$-recurrent then $B_w\oplus B_w^2\oplus ... \oplus B_w^r$ is $\E^*$-operator, for any $r\in \N$. We conclude the proof by Theorem \ref{theor.recurrence}. The proof of the unilateral version follows the same sketch of the bilateral one.

   \begin{remark}
 The converse of the unilateral version of Proposition \ref{bilateral.B_w.top.Erecurrent} does not hold. Denote by $B(x; r)$ the open ball centered at $x$ with radius $r$. Recall that an operator $T$ is called \emph{mixing} if the set $N(U, V)=\{n\in\N: T^n(U)\cap V\neq \emptyset\}$ is a cofinite set, for any non-empty open sets $U, V$ in $X$. It was pointed out to me (personal communication) by Quentin Menet the following:
 \begin{proposition}
 \label{result.Menet}
There exists a mixing weighted backward shift $B_w$ on $l^p(\Z_+)$ such that $N(x, B(e_0; 1/2))$ has upper Banach density equals to zero for any $x\in l^p(\Z_+)$.
 \end{proposition}
  Now, let $B_w$ given by Proposition \ref{result.Menet}, then $B_w$ is mixing on $l^p(\Z_+)$ and does not satisfy property $\p_{\mathcal{\overline{BD}}}$. Finally take into account the following fact proved in \cite{BMPS}: $B_w$ is mixing if and only if $B_w\oplus B_w^2\oplus ... \oplus B_w^r$ is mixing, for any $r\in \N$, thus the converse of the unilateral version of Proposition \ref{bilateral.B_w.top.Erecurrent} does not hold  \end{remark}
   The proof of Proposition \ref{result.Menet} is due to Quentin Menet and we include it here for the sake of completeness.
   \begin{proof} 
  Let $B_w$ a weighted backward shift such that $|w_{n}|\ge 1$ for any $n\ge 1$ and suppose there exists $x\in l^p(\N)$ and $m\ge 1$ such that 
\[\overline{Bd}\big(N\big(x,B\big(e_0; \frac{1}{2}\big)\big)\big)>\frac{1}{m}.\]
We denote by $A$ the set $B(e_0;1/2)$. We have thus
\[\lim_{s\to \infty}\limsup_{k\to \infty}\frac{|A\cap[k+1,k+s]|}{s}>\frac{1}{m}.\]
In other words, there exists $s_0\ge 1$ such that for any $s\ge s_0$, any $k_0\ge 1$, there exists $k\ge k_0$
such that
\[|A\cap[k+1,k+s]|>\frac{s}{m}.\]
In particular, we obtain the existence of an integer $l_0\ge 1$ such that for any $l\ge l_0$, we can find an integer $k\ge 1$ satisfying
\[|A\cap[k+1,k+lm]|>l.\]
This means that for any $l\ge l_0$, there exist $n_0,\cdots,n_l\in [1,lm]$ such that for any $0\le j\le l$,
\[\|B^{k+n_j}_wx-e_0\|^p_p<\frac{1}{2^p}.\]
We deduce that
\begin{equation}
\sum_{j=1}^l\prod_{\nu=1}^{k+n_0}|w_{n_j-n_0+\nu}x_{k+n_j}|^p<\frac{1}{2^p}
\label{eq:ceb1}
\end{equation}
and that for any $0\le j\le l$
\begin{equation}
\prod_{\nu=1}^{k+n_j}|w_{\nu} x_{k+n_j}|> \frac{1}{2}.
\label{eq:ceb2}
\end{equation}
We get by \eqref{eq:ceb2}
\[\sum_{j=1}^{l}\prod_{\nu=1}^{k+n_0}|w_{n_j-n_0+\nu} x_{k+n_j}|^p=\sum_{j=1}^{l}\frac{\prod_{\nu=1}^{k+n_j}|w_{\nu}|^p}{\prod_{\nu=1}^{n_j-n_0}|w_{\nu}|^p}|x_{k+n_j}|^p> \sum_{j=1}^{l}\frac{1}{2^p\prod_{\nu=1}^{n_j-n_0}|w_{\nu}|^p}\]
and thus by \eqref{eq:ceb1}
\[\inf_{1\le j\le lm} \frac{l}{\prod_{\nu=1}^{j}|w_{\nu}|^p}\le \inf_{1\le j \le l} \frac{l}{\prod_{\nu=1}^{n_j-n_0}|w_{\nu}|^p}\le 2^p\sum_{j=1}^{l}\frac{1}{2^p\prod_{\nu=1}^{n_j-n_0}|w_{\nu}|^p}<1.\]
Hence, we conclude that there exists $m\ge 1$ such that
\[\limsup_{l\to \infty}\frac{l}{\prod_{\nu=1}^{lm}|w_{\nu}|^p}< \infty\]
 because $|w_{n}|\ge 1$ for any $n\ge 1$.

Now, consider the weighted shift $B_w$ where $w_{\nu}=\big((\nu+1)/\nu\big)^{\frac{1}{2p}}$. Since
\[\prod_{\nu=1}^{n}|w_{\nu}|=(n+1)^{\frac{1}{2p}}\to \infty,\]
the weighted shift $B_w$ is mixing on $l^p(\Z_+)$, see Chapter 4 \cite{Grosse1}. On the other hand,  $N(x, B(e_0; 1/2))$ has upper Banach density equals to zero for any $x\in l_p(\Z_+)$ since for any $n\ge 1$, $|w_{n}|\ge 1$ and for any $m\ge 1$,
\[\frac{l}{\prod_{\nu=1}^{lm}|w_{\nu}|^p}=\frac{l}{\sqrt{lm+1}}\to \infty.\]
\end{proof}
 
\section{A characterization of reiteratively hypercyclic operators}
     Using the same ideas of the proof of Theorem \ref{theor.recurrence} we can obtain automatically more information about the return time set of a reiteratively hypercyclic operator respect to a reiteratively hypercyclic vector.
  \begin{definition}
   We will say that $T\in\Lin(X)$ is  \emph{$\E$-reiteratively hypercyclic with respect to $\lambda=(\lambda_n)_n$}  if there exists some $p\in\E$ and $x\in X$ such that for any non-empty open set $U$ in $X$ and $r\in\N$, it holds
 \[
\Big\{k\in\N: \overline{Bd}\Big(a\in\N: \lambda_aT^ax\in\cap_{j=0}^{r} T^{-jk}(U)\Big)>0 \Big\}\in p.
 \]
 \end{definition}
 In the case $(\lambda_n)_n=1$, we simply say that $T$ is $\E$-reiteratively hypercyclic.
 
 Following the same sketch of proof of Theorem \ref{theor.recurrence} we have the following:
  \begin{theorem}
   Let $(\lambda_n)_n$ a sequence of non-zero complex numbers and some $p\in\E$ such that there exists $A\in p$ for which
     \begin{equation}
    \mathcal{\underline{BD}}_1-\lim_n\Big|\frac{\lambda_n}{\lambda_{n+k}}\Big|=1, ~ \forall k\in A
     \label{cond.theor.equiv.rhc.op}
     \end{equation}
   then the family $(\lambda_nT^n)_n$ is reiteratively hypercyclic if and only if  $T$ is $\E$-reiteratively hypercyclic with respect to $\lambda=(\lambda_n)_n$.
    \end{theorem}

     As a consequence we obtain Theorem \ref{charact.rhc.Op}, which gives us a  characterization of reiteratively hypercyclic operators.
     
    \textbf{Proof of theorem \ref{charact.rhc.Op}.}
    
 When considering $(\lambda_n)_n=1$, condition \eqref{cond.theor.equiv.rhc.op} holds with $A=\N$. The proof of Theorem \ref{theor.recurrence} shows that in fact $\{k\in\N: \overline{Bd}(a\in\N: T^ax\in\cap_{j=0}^{r} T^{-jk}(U))>0\}\in \E^*$.
\begin{question}
\label{rel.propBD.rhc}
 Obviously, any reiteratively hypercyclic operator is hypercyclic and satisfies property $\p_{\mathcal{\overline{BD}}}$. We do not know anything about the converse. We wonder if this is true at least for weighted shifts. Observe that weighted shifts satisfying property $\p_{\mathcal{\overline{BD}}}$ are recurrent, therefore hypercyclic \cite{Costakis}. Hence, it is natural to ask the following question: does any weighted backward shift satisfying property $\p_{\mathcal{\overline{BD}}}$ is reiteratively hypercyclic?
 \end{question}
 \section*{Note Added in Proof}
 It should be noted that in Remark 3 (a) \cite{BMPP1}, it has been shown that any hypercyclic operator satisfying property $\p_{\mathcal{\overline{BD}}}$ is reiteratively hypercyclic. In particular, Question \ref{rel.propBD.rhc} can be answered affirmatively. Now, taking into account this fact, Theorem \ref{charact.rhc.Op} can be reformulated as follows.
  \begin{theorem}
  Let $T\in \Lin (X)$ a  hypercyclic operator. The following are equivalent:
  
  i) $T$ is reiteratively hypercyclic
  
  ii)  there exists some $x\in X$ such that for any non-empty open set $U$ in $X$ and any $r\in \N$, 
  \[
  \Big\{k\in \N: \overline{Bd}\Big(a\in \N: T^ax\in T^{-k}U\cap \cdots\cap T^{-rk}U\cap U\Big)>0\Big\}\in \E^*
  \]
  iii)   for any non-empty open set $U$ in $X$ there exists some $x\in X$ such that for  any $r\in \N$, 
  \[
  \Big\{k\in \N: \overline{Bd}\Big(a\in \N: T^ax\in T^{-k}U\cap \cdots\cap T^{-rk}U\cap U\Big)>0\Big\}\in \E^*.
  \]
  \end{theorem}

   \section*{Acknowledgement}
     We would like to thank Prof. Alfred Peris for his support as supervisor and also the referee for her or his careful comments which have greatly improved the presentation of this paper.

\end{document}